\UseRawInputEncoding 
\documentclass[11pt,onecolumn]{amsart}
\usepackage[pagewise]{lineno}
\newtheorem{theorem}{Theorem}[section]
\usepackage[left=1.5cm,top=1.5cm,bottom=1.5cm,right=1.5cm]{geometry}
\usepackage{latexsym}
\usepackage{amssymb,amsmath,amsfonts}

\newtheorem{lemma}[theorem]{Lemma}
\newtheorem{corollary}[theorem]{Corollary}

\theoremstyle{definition}
\newtheorem{definition}[theorem]{Definition}
\newtheorem{conclusion}[theorem]{Conclusion}

\newtheorem{example}[theorem]{Example}
\newtheorem{remark}[theorem]{Remark}
\newcommand{\ab}{|}

\newcommand{\R}{\mathbb R}

\newcommand{\N}{\mathbb N}

\newcommand{\Fcal}{\mathcal{F}}
\numberwithin{equation}{section}
 \global\parskip 6pt \oddsidemargin
0.1cm \evensidemargin 0.4cm \headheight 0.1cm \topmargin -2.0cm
\begin{document}

\pagenumbering{arabic}
  \title[On generalized $\Phi$-strongly monotone mappings and algorithms]
 {On generalized $\Phi$-strongly monotone mappings and algorithms for the solution of equations of Hammerstein type}
\author{ M.O. Aibinu$^{*}$ and O.T. Mewomo}
 \email{{\bf Corresponding author:}  $^*$moaibinu@yahoo.com/ mathewa@dut.ac.za}

 \keywords{Generalized ${\Phi}$-strongly monotone, Hammerstein equation, Strong convergence.\\
{\rm } {\it Mathematics Subject Classification}: 47H06, 47J05, 47J25, 47H09.\\
M.O. Aibinu, O.T. Mewomo, On generalized $\Phi$-strongly monotone mappings and algorithms for the solution of equations of Hammerstein type, International Journal of Nonlinear Analysis and Applications, https://doi.org/10.22075/ijnaa.2019.16797.1894.}

\begin{abstract}
In this paper, we consider the class of generalized $\Phi$-strongly monotone mappings and the methods of approximating a solution of equations of Hammerstein type. Auxiliary mapping is defined for nonlinear integral equations of Hammerstein type. The auxiliary mapping is the composition of bounded generalized $\Phi$-strongly monotone mappings which satisfy the range condition. Suitable conditions are imposed to obtain the boundedness and to show that the auxiliary mapping is a generalized $\Phi$-strongly which satisfies the range condition.  A sequence is constructed and it is shown that it converges strongly to a solution of equations of Hammerstein type. The results in this paper improve and extend some recent corresponding results on the approximation of a solution of equations of Hammerstein type.
\end{abstract}
 
\maketitle
\section{Introduction}
Let $E$ be a real normed linear space and $E^*$ denotes its corresponding dual space.  We denote the value of the functional $x^* \in E^*$ at $x \in E$ by $\left\langle x^*, x \right\rangle,$ domain of $A$ by $D(A),$ range of $A$ by $R(A)$ and $N(A)$ denotes the set of zeros of $A \left(\mbox{i.e.,} \ N(A) = \left\{x \in D(A) : 0 \in Ax\right\} = A^{-1}0\right).$   A multivalued mapping $A : E \rightarrow 2^{E^*}$ from $E$ into $2^{E^*}$ is said to be monotone if for each $x, y \in E$, the following inequality holds:
 $$\left\langle \mu -\nu, x-y\right\rangle \geq 0 \ \ \forall \ \ \mu \in Ax,\ \ \nu\in Ay.$$
A single-valued mapping $A:D(A)\subset E\rightarrow E^*$ is monotone if ${\left\langle Ax - Ay, x - y \right\rangle}\geq 0,~\forall~x,y\in D(A).$  For a linear mapping $A$, the above definition reduces to $\left\langle Au, u \right\rangle \geq 0 ~\forall ~u \in D(A).$  Multivalued mapping $A$ is said to be generalized $\Phi$-strongly monotone if there exists a strictly increasing function $\Phi: [0, \infty) \rightarrow [0, \infty)$ with $\Phi(0)=0$ such that for each $x, y \in D(A),$
$$\left\langle \mu -\nu, x-y\right\rangle \geq \Phi(\|x-y \|) \ \ \forall \ \ \mu \in Ax,\ \ \nu \in Ay.$$
Given that $H$ is a real Hilbert space, a mapping $A: H\rightarrow2^{H}$ is said to be monotone if for each $x, y\in H,$
$$\left\langle \mu -\nu, x-y\right\rangle \geq 0 \ \ \forall \ \ \mu \in Ax,\ \ \nu\in Ay.$$ 
Let $A$ be a monotone mapping defined on $H.$ It is well known (see e.g., Zeidler \cite{e1}) that many physically significant problems can be modelled by initial-value problems of the form 
\begin{equation}\label{e1}
u'(t) + Au(t) = 0, u(0) = u_{0}.
\end{equation}
 Heat, wave and Schr$\ddot{o}$dinger equations are typical examples where such evolution equations occur. At an equilibrium state (that is, if $u(t)$ is independent of $t$), then (\ref{e1}) reduces to
\begin{equation}\label{e2}
 Au = 0.
\end{equation}
 Therefore, considerable research efforts have been devoted, especially within the past 40 years or so, to methods of finding approximate solutions (when they exist) of (\ref{e2}).  One important generalization of (\ref{e2}) is the so-called equation of Hammerstein type (see, e.g., Hammerstein \cite{e2}), where a nonlinear integral equation of Hammerstein type is one of the form
\begin{equation}\label{e4}
u(x) + \int_{\Omega}k(x,y)f(y, u(y))dy = h(x),
\end{equation}
where $dy$ stands for a $\sigma$-finite measure on the measure space $\Omega$,
the kernel $k$ is defined on $\Omega \times \Omega$, $f$ is a real-valued function defined on $\Omega \times \R$ and is in general nonlinear, $h$ is a given function on $\Omega$ and $u$ is the unknown function defined on $\Omega$. 
Let $g$ be a function from $\Omega \times {\R}^n$ into $\R$. We denote by ${\Fcal}(X,Y)$, the set of all maps from $X$ to $Y$.  The Nemystkii operator associated to $g$ is the operator $N_g:{\Fcal}(\Omega,{\R}^n)\rightarrow {\Fcal}(\Omega,\R)$ defined by $$ u\mapsto N_g(u)$$ where $(N_gu)(x)=g\left(x,u(x)\right)~\forall~u \in {\Fcal}\left(\Omega,~{\R}^n \right),~\forall~x\in \Omega.$
 For simplicity, we shall write $N_gu(x)$ instead of $(N_gu)(x)$.
 \begin{example}
 Given a map $g : \R \times \R \rightarrow \R$ defined by $$g(x,s)=\ab s \ab ~\forall ~(x,s)\in \R \times \R,$$
 the Nemystkii operator associated to $g$ is the expression 
 $N_gu(x)=\ab u(x)\ab$ for any map $u:\R\rightarrow \R$ and for any $x\in \R$.
 \end{example}
\begin{example}
Given a map $g : \R \times \R \rightarrow \R$ defined by $$g(x,s)=xe^s ~\forall~ (x,s)\in \R \times \R,$$
the Nemystkii operator associated to $g$ is the expression 
$N_gu(x)=xe^{u(x)}$ for any map $u:\R\rightarrow \R$ and for any $x\in \R$.
\end{example}
\par Observe that by the continuity of $g$, $N_g$ maps the set of real-valued continuous function on $\Omega;$ $C(\Omega)$ into itself.  Moreover, it maps the set of real-valued measurable function into itself. Define the operator $K: {\Fcal}(\Omega, \R)\rightarrow {\Fcal}(\Omega, \R)$  by
$$Kv(x)=\int_{\Omega}k(x,y)v(y)dy \ \mbox{for almost all} \ x \in \Omega, $$
and the Nemystkii operator $F: {\Fcal}(\Omega, \R) \rightarrow {\Fcal}(\Omega, \R)$ associated with $f$ by 
$$Fu(x) = f(x, u(x)) \ \mbox{for almost all} \ x\in \Omega,$$
then the integral (\ref{e4}) can be put in functional equation form as follows:
\begin{equation}\label{e5}
u+KFu=0,
\end{equation}
where without loss of generality, we have taken $h\equiv 0$. Also, Hammerstein equations play crucial roles in solving several problems that arise in differential equations (see, e.g., Pascali and Sburlan \cite{e4}, Chapter IV, p. $164$) and applicable in theory of optimal control systems and in automation and network theory (see, e.g.,  Dolezale \cite{b14}).  Several authors have proved existence and uniqueness theorems for equations of the Hammerstein type (see, e.g., Br$\acute{e}$zis and  F. E. Browder (\cite{b15, b16, b22}); Browder and Gupta \cite{b38}; Chepanovich \cite{b19}; De Figueiredo and Gupta \cite{b20}).
\par Let $C$ be a nonempty closed convex subset of a real Banach space $E.$  A self-mapping $T : C\rightarrow C$ is said to be nonexpansive if $\|Tx-Ty\|\leq \|x-y\|~~\forall~~x,y\in C.$ If $E$ is smooth, $T: C \rightarrow E$ is said to be firmly nonexpansive type (see e.g., \cite{b32}), if 
$$\left\langle Tx-Ty, JTx-JTy\right\rangle \leq \left\langle Tx-Ty, Jx-Jy\right\rangle  \  \mbox{for all} \  x, y\in C,$$
where $J: E\rightarrow 2^{E^*}$  is the normalized duality mapping defined in Section 2.
 
\par For the iterative approximation of solutions of (\ref{e2}), the monotonicity of $A$ is
crucial. A mapping $A: E\rightarrow 2^{E^*}$ is said to be maximal monotone if it is monotone and $R(J+tA)$ is all of $E^*$ for some $t>0.$  Given that $A$ is monotone and $R(J+tA)=E^*$  for all $t>0,$ then $A$ is said to satisfy the range condition. Let $E$ be a uniformly smooth and uniformly convex Banach space and $A,$ a maximal monotone or (a monotone mapping which satisfies the range condition). Then, one can define for all $t>0,$ the resolvent $J_t:C\rightarrow D(A)$ by
$$J_tx=\left\{z\in E: Jx\in Jz+tAz\right\}$$
for all $x\in C,$ where $C$ is a closed convex subset of $E.$ The fact that $F(J_t)=A^{-1}0$ is well known where $F(J_t)$ is the set of fixed points of $J_t$ (see e.g., \cite{e35, b2, e31}). There exists some interesting reports on the class of monotone mappings (See e.g, \cite{ b9, aibinu, Chidume1, Djitte1, Yekini2}).
 
\par In this present work, it is shown that if $A$ is a multivalued generalized $\Phi$-strongly monotone mapping and such that $R(J_p+t_0A) = E^*$ for some $t_0 > 0,$ then $R(J_p+tA) = E^*$ for all $t > 0,$ where  $J_{p},~ p >1$ is the generalized duality mapping. That is, a maximal monotone mapping satisfies the range condition. Also, a strong convergence theorem for approximating a solution of equations of Hammerstein type is established. We consider the generalized $\Phi$-strongly monotone mapping which is the largest such that if a solution of the equation $0\in Ax$ exists, it is necessarily unique. Our results generalize and improve some important and recent results of Chidume and Idu \cite{b28}.

\section{Preliminaries}
Let $S := \left\{x \in E : \|x \| = 1\right\}$ denotes a unit sphere of a Banach space $E$ with dimension greater than or equal to two. The space $E$ is said to be G$\hat{a}$teaux differentiable (or is smooth) if the limit
$$\displaystyle \lim_{t\rightarrow 0}\frac{\|x +ty \|-\|x \|}{t}$$
exists for each $x, y \in S.$ If $E$ is smooth and the limit is attained uniformly for each $x, y \in S,$ then it is said to be uniformly smooth. A Banach space $E$ is said to be strictly convex if
$$\|x \|=\|y \|=1, x\neq y\Rightarrow \frac{\|x + y \|}{2}<1.$$
The space $E$ is said to be uniformly convex if, for each $\epsilon \in(0, 2],$ there exists a $\delta:= \delta(\epsilon) > 0$ such that for each $x, y \in S, ~\|x - y\|\geq \delta$ implies that $\frac{\|x-y\|}{2}\leq1-\delta.$ $E$ is reflexive if and only if the natural embedding of $E$ into $E^{**}$ is onto. It is known that a uniformly convex Banach space is reflexive and strictly convex. Also, if $E$ is a reflexive Banach space, then, it is strictly convex (respectively smooth) if and only if $E^*$ is smooth (respectively strictly convex).

\par Let $\varphi : [0, \infty) \rightarrow [0, \infty)$ be a strictly increasing continuous function such that $\varphi(0)= 0$ and $\varphi(t) \rightarrow \infty$ as $t \rightarrow \infty, ~\varphi$ is called a gauge function.  We associate to $\varphi,$ the duality mapping  $J_{\varphi}: E\rightarrow 2^{E^*}$ which is defined as
$$J_{\varphi}(x) =\left\{ f \in E^* :\left\langle x, f \right\rangle = \|x\|\| f\|, ~\| f\|=\varphi(\| x\|) \right\},$$
where $E^*$ denotes the dual space of $E$ and $\left\langle . , .\right\rangle$ denotes the generalized duality pairing. If $\varphi(t)=t^{p-1}, ~p>1,$ the duality mapping $J_{\varphi}=J_p$ is called generalized duality mapping. The duality mapping with guage $\varphi(t)=t$ (i.e. $p=2$) is denoted by $J$ and is referred to as the normalized duality mapping. It follows from the definition that $J_{\varphi}(x)=\frac{\varphi(\|x\|)}{\|x\|}J(x)$ for each $x\neq0$ and $J_p(x)={\|x\|}^{p-2}J(x), ~p>1.$  $J_{\varphi}$ is single-valued if $E$ is smooth and if $E$ is a reflexive strictly convex Banach space with strictly convex dual space $E^*,$ $J_{p} : E \rightarrow E^*$ and $J_{q} : E^* \rightarrow E$ being the duality mappings with gauge functions $\varphi(t)=t^{p-1}$ and $\varphi(s)=s^{q-1},$ $\frac{1}{p}+\frac{1}{q}=1,$ respectively, then $J_{p}^{-1}=J_{q}$.  For a Banach space $E$ and $E^*$ as its dual space, the following properties of the generalized duality mapping have also been established (see e.g., Alber and Ryazantseva \cite{e7}, Cioranescu \cite{b3}, p. 25-77, Xu and Roach \cite{b7}, Z$\check{a}$linescu  \cite{b37}): 
 \begin{itemize}\label{r1}
	\item [(i)]  If $E$ is smooth, then $J_p$ is single-valued and norm-to-weak$^*$ continuous;
	\item [(ii)] If $E$ is strictly convex, then $J_p$ is strictly monotone (injective, in particular, i.e, if $x \neq y,$ then $J_px \cap J_py =\emptyset$);
	\item [(iii)] If $E$ is reflexive, then $J_p$ is onto;
		\item [(iv)]The expression $\left\langle J_px, x\right\rangle$ is naturally regarded as having power $p$ as $\left\langle J_px, x\right\rangle={\|x\|}^p;$
	\item [(v)] If $E$ is uniformly smooth, then $J_q : E^* \rightarrow E$ is a generalized duality mapping on $E^*$, $J^{-1}_p = J_q$, $J_pJ_q = I_{E^*}$ and $J_qJ _p= I_{E}$, where $I_{E}$ and $I_{E^*}$ are the identity mappings on $E$ and $E^*$ respectively. 
	\end{itemize}
	\begin{definition}\label{d1}
Let $E$ be a smooth real Banach space with dual space $E^*,$ the followings were introduced by Aibinu and Mewomo \cite{b4}.
\begin{itemize}
\item[(i)] The function ${\phi}_p: E \times E \rightarrow \R$ is defined by
$${\phi}_p(x,y) =  \frac{p}{q}{\|x \|}^q - p\left\langle x, J_py \right\rangle + {\|y \|}^p , \ \mbox{for all} \ x,y \in E,$$
where $J_p$ is the generalized duality map from $E$ to $E^*$,  $p$ and $q$ are real numbers such that $q\geq p>1$ and $\frac{1}{p} + \frac{1}{q} =1$. Notice that taking $p=2$ in (\ref{e11}), it reduces to

$${\phi} (x,y)= {\|x \|}^2 - 2\left\langle x, Jy \right\rangle + {\|y \|}^2, \ \mbox{for all} \ x,y \in E,$$

which was introduced by Alber \cite{b1}.
\item[(ii)]The mapping $V_p: E\times E^* \rightarrow \R $ is defined by
$$V_p (x,x^*)= \frac{p}{q}{\|x \|}^q - p\left\langle x, x^* \right\rangle + {\|x^* \|}^p  ~~ \forall ~~ x \in E, x^*\in E^*  \ \mbox{such that} \ q\geq p>1,\ \ \frac{1}{p} + \frac{1}{q} =1.$$
\end{itemize}	
\end{definition}

\begin{remark}
These remarks follow from Definition \ref{d1}:
\begin{itemize}
\item [(i)] It is obvious from the definition of the function ${\phi}_p$ that
\begin{equation}\label{e11}
(\|x\|- \|y \|)^p \leq {\phi}_p(x,y) \leq (\|x\| + \|y \|)^p  \ \mbox{for all} \ x,y \in E.
\end{equation}
	\item [(ii)] Clearly, we also have that
\begin{equation}\label{e20}
V_p (x,x^*)= {\phi}_p (x,J^{-1}x^*)~~ \forall ~~ x \in E,\ \ x^*\in E^*.
\end{equation}
\end{itemize}
\end{remark}
\par In the sequel, we shall need the following lemmas.
\begin{lemma}\label{l20}
Aibinu and Mewomo \cite{b4}. Let $E$ be a smooth uniformly convex real Banach space with $E^*$ as its dual. Then
\begin{equation}\label{e21}
V_p (x, x^*) + p\left\langle J^{-1}x^*-x, y^* \right\rangle \leq V_p(x, x^*+y^*)
\end{equation}
for all $x\in E$ and $x^*, y^* \in E^*$.
\end{lemma}
\begin{lemma}\label{l21}
Aibinu and Mewomo \cite{b4}. Let $E$ be a smooth uniformly convex real Banach space. For $d > 0$, let $ B_d(0):= \left\{ x \in E: \| x \| \leq d \right\} $. Then for arbitrary $x, y \in B_d(0)$,
$$ {\|x-y \|}^p \geq{\phi}_p (x,y)- \frac{p}{q}{\| x\|}^q,~~q\geq p>1,~~ \frac{1}{p}+\frac{1}{q}=1.$$
\end{lemma}

\begin{lemma}\label{l31}
Aibinu and Mewomo \cite{b4}. Let $E$ be a reflexive strictly convex and smooth real Banach space with the dual $E^*$. Then
\begin{equation}\label{e31}
{\phi}_p(y,x)-{\phi}_p(y,z)\geq p\left\langle z-y, Jx-Jz\right\rangle  \ \mbox{ for all} \ x, y, z\in E.
\end{equation}
\end{lemma}
\begin{lemma}\label{l11}
 Xu \cite{bh1}. Let $\left\{a_n\right\}$ be a sequence of nonnegative real numbers satisfying the following
relations:
$$a_{n+1}\leq(1-{\alpha}_n)a_n+{\alpha}_n{\sigma}_n+{\gamma}_n, ~~n\in \N,$$
where
\begin{itemize}
\item[(i)]${\left\{\alpha\right\}}_n\subset (0, 1)$, $\displaystyle\sum_{n=1}^{\infty} {\alpha}_n = \infty$;
\item[(ii)]$\limsup {\left\{\sigma\right\}}_n\leq 0$;
\item[(iii)] ${\gamma}_n\geq 0$, $\displaystyle\sum_{n=1}^{\infty} {\gamma}_n < \infty$.
\end{itemize}
 Then, $a_n\rightarrow 0$ as $n\rightarrow \infty$.
\end{lemma}
\begin{lemma}\label{e18}
Chidume and Idu \cite{b28}. For a real number $p>1$, let $X, Y$ be real uniformly convex and uniformly smooth Banach spaces. Let $W:=X\times Y$ with the norm ${\|w\|}_W=\left({\|u\|}^p_X+{\|v\|}^p_Y \right)^{\frac{1}{p}}$ for arbitrary $w:=(u,v)\in W$. Let $W^*:=X^*\times Y^*$ denotes the dual space of $Z$. For arbitrary $z=(u,v)\in Z$, define the map $j_p^Z:Z\rightarrow Z^*$ by
$$j_p^W(z)=j_p^W(u,v)=\left(j_p^X(u),j_p^Y(v)\right),$$
such that for arbitrary $w_1=(u_1,v_1)$, $w_2=(u_2,v_2)$ in $Z$, the duality pairing $\left\langle.,.\right\rangle$ is given by
$$\left\langle w_1,j_p^W(w_2)\right\rangle=\left\langle u_1,j_p^X(u_2)\right\rangle+\left\langle v_1,j_p^Y(v_2)\right\rangle.$$
Then,
\begin{itemize}
	\item[(i)]$W$ is uniformly smooth and uniformly convex,
	\item[(ii)]	$j_p^W$ is single-valued duality mapping on $W$.
\end{itemize}
\end{lemma}
\begin{lemma}\label{l17}
Chidume and Idu \cite{b28}. Let $E$ be a uniformly convex and uniformly smooth real Banach space. Let $F : E \rightarrow E^*$ and $K : E^* \rightarrow E$ be monotone mappings with $D(F) = R(K) = E$. Let $T : E \times E^* \rightarrow E^* \times E$ be defined by $T(u, v) = (Ju-Fu+v, J^{-1}v-Kv-u)$ for all $(u, v) \in E \times E^*$, then $T$ is $J$-pseudocontractive. Moreover, if the Hammerstein equation $u + KFu = 0$ has a solution in $E$, then $u^*$ is a solution of $u + KFu = 0$ if and only if $(u^*, v^*)\in  F_E^ J (T)$, where $v^* = Fu^*$.
\end{lemma}
\begin{lemma}\label{l12}
Z$\check{a}$linescu  \cite{b37}. Let $\psi : {\R}^+\rightarrow {\R}^+$ be increasing with $\displaystyle \lim_{t\rightarrow \infty} \psi(t)=\infty.$ Then $J^{-1}_{\psi}$ is single-valued and uniformly continuous on bounded sets of $E^*$ if and only if $E$ is a uniformly convex Banach space.
\end{lemma}
\begin{theorem}\label{t3}
Xu \cite{d26}. Let $E$ be a real uniformly convex Banach space. For arbitrary $r>0$, let $B_r(0):=\left\{x \in E: \|x\| \leq r\right\}$. Then, there exists a continuous strictly increasing convex function
$$g:[0,\infty)\rightarrow [0,\infty),~~ g(0)=0,$$
such that for every $x, y \in B_r(0), j_p(x)\in J_p(x), j_p(y)\in J_p(y)$, the following inequalities hold:
 \begin{itemize}
	\item[(i)]${\|x + y \|}^p \geq {\|x \|}^p + p\left\langle y, j_p(x) \right\rangle + g(\|y \|)$;
	\item[(ii)]$\left\langle x-y, j_p(x)-j_p(y) \right\rangle \geq g(\|x-y \|)$.
 \end{itemize}	
\end{theorem}
\begin{lemma}\label{l13}
B. T. Kien \cite{b34}. The dual space $E^*$ of a Banach space $E$ is uniformly convex if and only if the duality mapping $J_p$ is a single-valued map which is uniformly continuous on each bounded subset of $E$.
\end{lemma}
\begin{lemma}\label{l15}
Kamimura and Takahashi \cite{r16}. Let $E$ be a smooth uniformly convex real Banach space and let $\left\{x_n\right\}$ and $\left\{y_n\right\}$ be two sequences from $E.$ If either $\left\{x_n\right\}$ or $\left\{y_n\right\}$ is bounded and $\phi (x_n, y_n) \rightarrow 0$ as $n \rightarrow \infty$, then $ \| x_n - y_n \| \rightarrow 0$ as $n \rightarrow \infty$.
\end{lemma}
\begin{theorem}\label{t8}
Kido \cite{b30}. Let $E^*$ be a real strictly convex dual Banach space with a Fr$\acute{e}$chet differentiable norm and $A$ a maximal monotone operator from $E$ into $E^*$ such that $A^{-1}0\neq \emptyset$. Let $J_tx:=(J+tA)^{-1}x$ be the resolvent of $A$ and $P$ be the nearest point retraction of $E$ onto $A^{-1}0$. Then, for every $x\in E$, $J_tx$ converges strongly to $Px$ as $t\rightarrow \infty$.
\end{theorem}

\section{Main Results}
 We give and prove the following lemmas which are useful in establishing our main result.
\begin{lemma}\label{e6} 
Suppose $E$ is a Banach space with the dual $E^*.$ Let $F : E\rightarrow E^*$ and $K : E^*\rightarrow E$ be mappings such that $D(K)=R(F)$ and the following conditions hold:
 \begin{itemize}
	\item[(i)] For each $u_1, u_2\in E,$ there exists a strictly increasing function 	${\Phi}_1: [0,\infty)\rightarrow [0,\infty)$ with ${\Phi}_1(0)=0$ such that
	$$\left\langle Fu_1-Fu_2, u_1-u_2\right\rangle \geq {\Phi}_1(\|u_1-u_2\|);$$
	\item[(ii)]For each $v_1, v_2\in E^*,$ there exists a strictly increasing function 	${\Phi}_2: [0,\infty)\rightarrow [0,\infty)$ with ${\Phi}_2(0)=0$ such that
	$$\left\langle Ku_1-Ku_2, v_1-v_2\right\rangle \geq {\Phi}_2(\|v_1-v_2\|);$$
	\item[(iii)]${\Phi}_i(t)\geq r_it$ for $t\in [0,\infty)$ and $r_i>0, ~i=1,2.$
 \end{itemize}
 Let $W:=E\times E^*$ with norm $\|w\|_W:=\|u\|_E+\|v\|_{E^*}$ for $w=(u, v)\in W.$ Define a mapping $A: W\rightarrow W^*$ by $Aw:=(Fu-v, u+Kv).$ 
  \begin{itemize}
	\item[(i)]Then for each $w_1, w_2\in W,$ there exists a strictly increasing function $\Phi: [0,\infty)\rightarrow [0,\infty)$ with $\Phi(0)=0$ such that
$$\left\langle Aw_1-Aw_2, w_1-w_2\right\rangle\geq \Phi(\|w_1-w_2\|);$$
	\item[(ii)] Suppose that $F$ and $K$ are bounded mappings, then $A$ is a bounded map.
\end{itemize}	
\end{lemma}
\begin{proof}
\begin{itemize}
	\item[(i)] Define $\Phi: [0,\infty)\rightarrow[0,\infty)$ by $\Phi(t):=\min\left\{r_1, r_2 \right\}t$ for each $t\in [0, \infty).$ Clearly, $\Phi$ is a strictly increasing function with $\Phi(0)=0.$ For $w_1 = (u_1, v_1), ~~w_2 = (u_2, v_2) \in W$, we have $Aw_1 = \left(Fu_1 - v_1, Kv_1 + u_1 \right)$ and  $Aw_2 = \left(Fu_2 - v_2, Kv_2 + u_2 \right)$ such that
 $$ Aw_1 - Aw_2 = \left( Fu_1 - Fu_2 - (v_1 -v_2), Kv_1 - K v_2 + (u_1 - u_2)  \right).$$ 
 Therefore, the following estimate follows from the properties of $F$ and $K.$
 \begin{eqnarray*}
 \left\langle Aw_1 - Aw_2, w_1 - w_2  \right\rangle &=& \left\langle Fu_1 - Fu_2 - (v_1 -v_2), u_1 -u_2 \right\rangle \\
& &+ \left\langle Kv_1 - K v_2 + (u_1 - u_2), v_1 -v_2 \right\rangle \\
&=& \left\langle Fu_1 - Fu_2, u_1 -u_2\right\rangle - \left\langle v_1 -v_2, u_1 -u_2 \right\rangle \\
& &+  \left\langle Kv_1 - K v_2, v_1 -v_2 \right\rangle +  \left\langle u_1 - u_2, v_1 -v_2 \right\rangle \\
&\geq &{\Phi}_1(\|u_1-u_2\|)+ {\Phi}_2(\|v_1-v_2\|)\\
&\geq &r_1\|u_1-u_2\|+ r_2\|v_1-v_2\|\\
&\geq & \min \left\{r_1, r_2\right\}(\| u_1-u_2 \|+\| v_1-v_2 \|) \\
&= & \Phi(\| w_1-w_2 \|).
\end{eqnarray*}
	\item[(ii)] By the definition of $A,$ it is a bounded map since $F$ and $K$ are bounded mappings.
\end{itemize}
\end{proof}
\begin{remark}\label{e7}
Recall that a mapping $A : E \rightarrow E^*$ is said to be strongly monotone if there exists a constant $k\in(0, 1)$ such that
 $$\left\langle Ax -Ay, x-y\right\rangle \geq k{\|x-y \|}^2 \ \ \forall \ \ x, y \in D(A).$$
 Therefore for a strongly monotone mapping, it is required that the norm on $W$ be defined as $$\|w\|_W^2:={\| u \|}_E^2+{\| v\|}^2_{E^*}.$$
\end{remark}

\par An analogue of Lemma {\bf 2.5}, Chidume and Djitte, \cite{e33}, which was proved in a Hilbert space is given below in a uniformly smooth and uniformly convex Banach space.

\begin{lemma}\label{e8}
Let $E$ be a  uniformly smooth and uniformly convex Banach space with dual $E^*.$ Suppose $D(A)=E$ and $A: E\rightarrow 2^{E^*}$ is a multivalued generalized $\Phi$-strongly monotone mapping such that $R(J_p+t_0A)=E^*$ for some $t_0>0.$ Then $A$ satisfies the range condition, that is, $R(J_p + tA)=E^*$ for all $t >0.$
\end{lemma}
\begin{proof}
By the strict convexity of $E,$ we obtain for every $x \in E,$ there exist unique $x_{t_0} \in E$ and such that
$$J_px \in J_px_{t_0} + t_0 Ax_{t_0}.$$
Taking $J_{p_{_{{t}_0}}}( x )= x_{t_0}$, one can define a single-valued mapping $J_{p_{_{{t}_0}}} : E\rightarrow D(F)$ by $$J_{p_{_{{t}_0}}}:= (J_p + t_0 A)^{-1}J_p.$$  
$J_{p_{_{{t}_0}}}$ is called the resolvent of $A$. It is known that $(J_p + t_0A)$ is a bijection since it is monotone and $R(J_p+t_0A)=E^*.$   Since $E$ is a smooth and strictly convex Banach space and $A: E \rightarrow 2^{E^*}$ is such that $R(J_p+t_0A)=E^*$, for each $t_0>0$, one can verify that the resolvent $J_{p_{_{{t}_0}}}$ of $A,$ defined by 
$$J_{p_{_{{t}_0}}}(x)=\left\{z\in E:J_px\in J_pz+ t_0 Az \right\}=\left\{\left(J_p+t_0A\right)^{-1}J_px\right\}$$
for all $x\in E$ is a firmly nonexpansive type map. Infact, for $x_1, x_2\in E$ and $t_0 >0,$ and for every $J_{p_{_{{t}_0}}}(x_1), ~ J_{p_{_{{t}_0}}}(x_2) \in D(F),$ we have that $\frac{J_px_1-J_p(J_{p_{_{{t}_0}}}(x_1))}{t_0}, ~\frac{J_px_2-J_p(J_{p_{_{{t}_0}}}(x_2))}{t_0} \in A,$ and generalized $\Phi$-strongly monotonicity property of $A$ gives, 
 \begin{eqnarray*}
&&\left\langle \frac{J_px_1-J_p(J_{p_{_{{t}_0}}}(x_1))}{t_0}-\frac{J_px_2-J_p(J_{p_{_{{t}_0}}}(x_2))}{t_0}, ~~ J_{p_{_{{t}_0}}}(x_1)-J_{p_{_{{t}_0}}}(x_2)\right\rangle \geq\\
 &&{\Phi}(\|J_{p_{_{{t}_0}}}(x_1)-J_{p_{_{{t}_0}}}(x_2)\|)\geq0.
\end{eqnarray*}
 Consequently, 
  \begin{eqnarray}\label{e36}
&&\left\langle J_p(J_{p_{_{{t}_0}}}(x_1))- J_p(J_{p_{_{{t}_0}}}(x_2)), ~~ J_{p_{_{{t}_0}}}(x_1)-J_{p_{_{{t}_0}}}(x_2)\right\rangle \leq \nonumber\\ 
&& \left\langle J_px_1-J_px_2, ~~ J_{p_{_{{t}_0}}}(x_1)-J_{p_{_{{t}_0}}}(x_2)\right\rangle.
 \end{eqnarray}
Thus, the resolvent $J_{p_{_{{t}_0}}}$ is a  firmly nonexpansive type map. A simple computation from (\ref{e36}) shows that for $x, y \in E,$
 \begin{eqnarray}\label{ee60}
\|J_{p_{_{{t}_0}}}(x)-J_{p_{_{{t}_0}}}(y)|\leq \|x-y\|.
 \end{eqnarray}
We claim that
 $$R(J_p+t A)=E^*$$
 for any $t > \frac{{t}_0}{2}.$ Indeed, let $t > \frac{{t}_0}{2},$ for every $x\in E,$ we solve the equation
\begin{equation}\label{e9}
J_px+t Ax = w^*, ~x^*\in E^*.
\end{equation}
Notice that $x\in E$ is a solution of (\ref{e9}) provided that 
 \begin{eqnarray*}
J_px+t_0Ax = \frac{t_0}{t}w^*+(1-\frac{t_0}{t})J_px,
 \end{eqnarray*}
 which is equivalent to
$$x = J_{p_{_{{t}_0}}}\left(\frac{t_0}{t}w^*+(1-\frac{t_0}{t})J_px\right).$$
 By the contraction mapping principle, Eq.(\ref{e9}) has a unique solution since $|1-\frac{{t}_0}{t}|<1$ and this justifies the claim. It is given that $A$ is a monotone mapping and $R(J_p+{t}_0A)=E^*$ for some ${t}_0>0.$ By the claim, it follows that $R(J_p+ tA)=E^*$ for any $t>\frac{{t}_0}{2}.$ By induction, we therefore have that $R(J_p+tA)=E^*$ for any $t>\frac{{t}_0}{2^n}$ and any $n\in \N.$ Thus, $R(J_p+tA)=E^*$ for any $t>0.$
\end{proof}

\begin{lemma}\label{e13} 
Let $E$ be a uniformly smooth and uniformly convex real Banach space and denote the dual space by $E^*.$ Suppose $F: E\rightarrow E^*$ is a generalized ${\Phi}_1$-strongly monotone mapping such that $R(J_p+{t}_1F)=E^*$ for all ${t}_1>0$ and $K: E^*\rightarrow E$ is a generalized ${\Phi}_2$-strongly monotone mapping such that $R(J_q+{t}_2K)=E$ for all ${t}_2>0$. Let $W := E\times E^*$ with norm ${\|w \|}_W := {\| u\|}_{E} + {\| v \|}_{E^*} ~ \forall ~ w = \left(u, v\right) \in W$ and define a map $A: W\rightarrow W^*$ by
\begin{equation}
Aw = \left(Fu - v, Kv + u \right), \forall ~ w=(u, v)\in W,
\end{equation}
then $R(J_p + tA)=W^*$ for all $t>0.$
\end{lemma}

\begin{proof}
 We show that $R(J_p + tA)= W^*$ for all $t > 0$. Indeed, let ${t}_0$ be such that $0 < {t}_0 < 1$.
Denote the resolvents $J_{p_{_{{t}_0}}} : E\rightarrow D(F)$ of $F$ by $J_{p_{_{{t}_0}}}:= (J_p + {t}_0 F)^{-1}J_p$ and $J_{q_{_{{t}_0}}}: E^* \rightarrow D(K)$ of $K$ by  $J_{q_{_{{t}_0}}} = (J_q + {t}_0 K)^{-1}J_q.$
$J_{p_{_{{t}_0}}}$ and $J_{q_{_{{t}_0}}}$ are firmly nonexpansive type maps and hence (\ref{ee60}) holds.
Therefore, for $h :=(h_1, h_2) \in X^*$, define $G: W \rightarrow W$ by 
$$Gw = \left(J_{p_{_{{t}_0}}}(h_2 - {t}_0 u),~~ J_{q_{_{{t}_0}}}(h_1+ {t}_0 v)\right), \forall ~~ w = (u, v)\in W.$$
From the fact that (\ref{ee60}) holds for $J_{p_{_{{t}_0}}} $ and $J_{q_{_{{t}_0}}},$  we have
$$\| Gw_1 - Gw_2 \| \leq {t}_0 \| w_1 -w_2 \|~~~~ \forall~~~~ w_1, w_2 \in W.$$
Therefore $G$ is a contraction and by Banach contraction mapping principle, $G$ has a unique fixed point $w^*:=(u^*, v^*)\in W$, that is $Gw^* = w^*$ or equivalently $u^* = J_{p_{_{{t}_0}}}(h_2 - {t}_0 u^*), v^* =  J_{q_{_{{t}_0}}}(h_1+ {t}_0 v^*)$. These imply that $(J_p + {t}_0 A)w = h$. Lemma \ref{e6} gives that $A$ is a generalized $\Phi$-strongly monotone mapping and by Lemma \ref{e8}, $R(J_p + tA)= W^*$ for all $t> 0$.
\end{proof}

\begin{theorem}\label{e14}
Let $E$ be a uniformly smooth and uniformly convex real Banach space and denote the dual space by $E^*$. Let $F: E\rightarrow E^*$ be a generalized ${\Phi}_1$-strongly monotone mapping such that $R(J_p+{t}_1F)=E^*$ for all ${t}_1>0$ and $K: E^*\rightarrow E$ be a generalized ${\Phi}_2$-strongly monotone mapping such that $R(J_q+{t}_2K)=E$ for all ${t}_2>0$. Suppose $F$ and $K$ are bounded mappings such that $D(K) = R(F ) = E^*$. Define $\left\{u_n\right\}$ and $\left\{v_n\right\}$ iteratively for arbitrary $u_1 \in E$ and $v_1\in E^*$ by
\begin{equation}
u_{n+1} = J_q\left(J_pu_n - {\lambda}_n\left(Fu_n-v_n+{\theta}_n(J_pu_n-J_pu_1)\right)\right), n \in \N,
\end{equation}
\begin{equation}
v_{n+1} = J_p\left(J_qv_n - {\lambda}_n\left(Kv_n+u_n+{\theta}_n(J_qv_n-J_qv_1)\right)\right), n \in \N,
\end{equation}
where $J_p$ is the generalized duality mapping from $E$ to $E^*$ and $J_q$ is the generalized duality mapping from $E^*$ to $E.$ Let the real sequences $\left\{ {\lambda}_n\right\}$ and  $\left\{ {\theta}_n\right\}$ in $(0,1)$ be such that,
\begin{itemize}
	\item[(i)] $\lim {\theta}_n =0$ and $\left\{ {\theta}_n\right\}$ is decreasing;
	\item[(ii)]$ \displaystyle\sum_{n=1}^{\infty} {\lambda}_n{\theta}_n=\infty$;
	\item[(iii)]$\displaystyle \lim_{n\rightarrow \infty}\left(({\theta}_{n-1}/{\theta}_n)-1\right)/{{\lambda}_n{\theta}_n}=0,~~ \displaystyle\sum_{n=1}^{\infty} {\lambda}_n <\infty$.
\end{itemize}	
 Suppose that $u + KFu=0$ has a solution in $E$. There exists a real constant ${\gamma}_0>0$ with ${\psi}({\lambda}_nM)\leq{\gamma}_0, \ \ n \in \N$ for some constant $M>0$. Then, the sequence $\left\{u_n \right\}$ converges strongly to the solution of $0=u+KFu$.
\end{theorem}

\begin{proof}
Let $W := E\times E^*$ with norm ${\|x \|}_W^p :=   {\| u\|}_{E}^p + {\| v \|}_{E^*}^p   ~ \forall ~ w = \left(u, v\right) \in W$ and define ${\wedge}_p : W\times W \rightarrow \R$ by
$$ {\wedge}_p (w_1 , w_2) = {\phi}_p (u_1, u_2) + {\phi}_p (v_1, v_2),$$
where respectively $w_1=(u_1, v_1)$ and $w_2=(u_2, v_2)$. Let $u^*\in E$ be a solution of $u+KFu = 0$. Observe that setting $v^*:=Fu^*$ and $w^*:=(u^*, v^*)$, we have that $u^* = -Kv^*$.

 We divide the proof into two parts.\\
{\bf Part 1:} We prove that $\left\{w_n \right\}$ is bounded, where $w_n := (u_n, v_n).$ Let $r > 0$ be sufficiently large such that
\begin{equation}\label{e37}
\Phi(\frac{\delta}{2})\geq r\geq \max \left\{4{\wedge}_p(w^*, w_1), ~{\delta}^p + \frac{p}{q}{\| x^*\|}^q \right\},
\end{equation}
where $\delta$ is a positive real number and $\Phi:=\min\left\{{\Phi}_1, {\Phi}_2\right\}.$ The proof is by induction. By construction, ${\wedge}_p(w^*, w_1)\leq r$. Suppose that ${\wedge}_p(w^*, w_n)\leq r$ for some $n\in \N$. We show that ${\wedge}_p(w^*, w_{n+1})\leq r.$ Suppose this is not the case, then ${\wedge}_p(w^*, w_{n+1})> r.$\\
 From inequality (\ref{e11}), we have $\|w_n \| \leq r^{\frac{1}{p}} + \|w^* \|$. Let $B:=\left\{w\in E: {\wedge}_p(w^*, w)\leq r \right\}$ and notice that by Lemma (\ref{l12} and \ref{l13}),  $J_q$ and $J_p$ are uniformly continuous on bounded subsets. Consequently, since $F$ and $K$ are bounded, we define
\begin{equation}\label{e38}
M_1 :=\sup \left\{ {\| Fu+{\theta}_n(J_pu-J_pu_1)\|}: {\theta}_n\in(0,1), u \in B\right\}+1,
\end{equation}
\begin{equation}\label{e39}
M_2 :=\sup \left\{ {\| Kv+{\theta}_n(J_pv-J_pv_1)\|}: {\theta}_n\in(0,1), v \in B\right\}+1.
\end{equation}
Let ${\psi}_1: [0,\infty)\rightarrow [0,\infty)$ be the modulus of continuity of $J_q$ and ${\psi}_2: [0,\infty)\rightarrow [0,\infty)$ be the modulus of continuity of $J_p.$ Recall that by the uniform continuity of $J_q$ and $J_p$  on bounded subsets of $E^*$ and $E$ respectively. Then we have
\begin{eqnarray}\label{e40}
\|J_q(J_pu_n)-  J_q( J_pu_n - {\lambda}_n\left(Fu_n+{\theta}_n(J_pu_n-J_pu_1)\right)) \| &\leq&{\psi}_1({\lambda}_nM_1),
\end{eqnarray}
\begin{eqnarray}\label{e41}
\|J_p(J_qv_n)-  J_p( J_qv_n - {\lambda}_n\left(Kv_n+{\theta}_n(J_qv_n-J_qv_1)\right)) \| &\leq&{\psi}_2({\lambda}_nM_2).
\end{eqnarray}
 Let $M_0:=M_1+M_2,$ since $\Phi:=\min\left\{{\Phi}_1, {\Phi}_2\right\},$ one can define $${\gamma}_0:=\min\left\{1, \frac{\Phi(\frac{\delta}{2})}{2M_0} \right\} \ \mbox{where} \ {\psi}({\lambda}_nM_0) \leq{\gamma}_0 \ \mbox{with} \ {\psi}({\lambda}_nM_0) \geq \frac{\delta}{2},$$
 and ${\psi}:={\psi}_1+{\psi}_2.$
Applying Lemma \ref{l20} with $y^* := {\lambda}_n \left(Fu_n+{\theta}_n(J_pu_n-J_pu_1)\right)$ and by using the definition of $u_{n+1}$, we compute as follows,

 \begin{eqnarray*}
{\phi}_p(u^*, u_{n+1})
 & = & {\phi}_p\left(u^*, J_q\left(J_pu_n - {\lambda}_n\left(Fu_n+{\theta}_n(J_pu_n-J_pu_1)\right)\right)\right)\\
 & = & V_p\left(u^*, J_pu_n - {\lambda}_n\left(Fu_n+{\theta}_n(J_pu_n-J_pu_1)\right)\right)\\
 & \leq & V_p(u^*, J_pu_n )\\
 &&-p{\lambda}_n\left\langle J_q( J_pu_n - {\lambda}_n\left(Fu_n+{\theta}_n(J_pu_n-J_pu_1)\right))-u^*, Fu_n+{\theta}_n(J_pu_n-J_pu_1)\right\rangle \\
 &= & {\phi}_p(u^*, u_n )-p{\lambda}_n \left\langle u_n-u^*, Fu_n+{\theta}_n(J_pu_n-J_pu_1)\right\rangle\\
 & &-p{\lambda}_n\left\langle J_q( J_pu_n - {\lambda}_n\left(Fu_n+{\theta}_n(J_pu_n-J_pu_1)\right))-u_n, Fu_n+{\theta}_n(J_pu_n-J_pu_1)\right\rangle.
 \end{eqnarray*}
 By Schwartz inequality and uniform continuity property of $J_q$ on bounded sets of $E^*$ (Lemma \ref{l12}), we obtain
\begin{eqnarray*}
{\phi}_p(u^*, u_{n+1})
& \leq & {\phi}_p(u^*, u_n )-p{\lambda}_n \left\langle u_n-u^*, Fu_n+{\theta}_n(J_pu_n-J_pu_1)\right\rangle \\
& & + p{\lambda}_n{\psi}_1({\lambda}_nM_1)M_1 \ \mbox{(By applying inequality (\ref{e40}))}\\
& \leq &{\phi}_p(u^*, u_n )-p{\lambda}_n \left\langle u_n-u^*,Fx_n-Fu^*\right\rangle  \ \mbox{since} \ u^*\in N(F)) \\
& &-p{\lambda}_n{\theta}_n \left\langle u_n-u^*, J_pu_n-J_pu_1\right\rangle +p{\lambda}_n{\psi}_1({\lambda}_nM_1)M_1.
\end{eqnarray*}
By Lemma \ref{l31}, $p\left\langle u_n-u^*, J_pu_1-J_pu_n\right\rangle \leq {\phi}_p(u^*, u_1 )-{\phi}_p(u^*, u_n )\leq {\phi}_p(u^*, u_1 ).$ Also, since $F$ is generalized $\Phi$-strongly monotone, we have,
\begin{eqnarray} \label{e42}
{\phi}_p(u^*, u_{n+1})
 &\leq&  {\phi}_p(u^*, u_n )-p{\lambda}_n{\Phi}_1({\|u_n -u^* \|})\nonumber\\
 && +p{\lambda}_n{\theta}_n \left\langle u_n-u^*, J_pu_1-J_pu_n\right\rangle +p{\lambda}_n{\psi}_1({\lambda}_nM_1)M_1 \nonumber \\
 &\leq&  {\phi}_p(u^*, u_n )-p{\lambda}_n{\Phi}_1({\|u_n -u^* \|}) +p{\lambda}_n{\theta}_n{\phi}_p(u^*, u_1 ) +p{\lambda}_n{\psi}_1({\lambda}_nM_1)M_1.
  \end{eqnarray}
  By the uniform continuity property of $J_q$ on bounded sets of $E^*$, we have
  $$\|u_{n+1}-u_n\|=\|J_q(J_pu_{n+1})-J_q(J_pu_n)\|\leq{\psi}_1({\lambda}_nM_1),$$
  such that
  $$\|u_{n+1}-u^*\|-\|u_n-u^*\|\leq{\psi}_1({\lambda}_nM_1),$$
  which gives
  \begin{eqnarray}\label{e43}
  \|u_n-u^*\|&\geq& \|u_{n+1}-u^*\|-{\psi}_1({\lambda}_nM_1).
  \end{eqnarray}
  From Lemma \ref{l21},
  \begin{eqnarray*}
  {\|u_{n+1}-u^*\|}^p&\geq&{\phi}_p(u^*, u_{n+1})-\frac{p}{q}\|u^*\| \\
  &\geq&r-\frac{p}{q}\|u^*\| \\
  &\geq&\left({\delta}^p+\frac{p}{q}\|u^*\|\right)-\frac{p}{q}\|u^*\| \\
  &\geq&{\delta}^p.
  \end{eqnarray*}
  So, $$\|u_{n+1}-u^*\|\geq \delta.$$
  Therefore, the inequality (\ref{e43}) becomes,
 \begin{eqnarray*}  
  \|u_n-u^*\|&\geq&\delta-{\psi}_1({\lambda}_nM_1)\\
  &\geq&\frac{\delta}{2}.
 \end{eqnarray*}
 Thus,
   \begin{eqnarray}\label{e44}
 {\Phi}_1({\|u_n -u^* \|})&\geq& {\Phi}_1(\frac{\delta}{2}).
  \end{eqnarray}
  Substituting (\ref{e44}) into (\ref{e42}) gives
  \begin{eqnarray}\label{e45}
{\phi}_p(u^*, u_{n+1})
 &\leq&  {\phi}_p(u^*, u_n )-p{\lambda}_n {\Phi}_1(\frac{\delta}{2})+p{\lambda}_n{\theta}_n{\phi}_p(u^*, u_1 )\nonumber\\
 &&+p{\lambda}_n{\psi}_1({\lambda}_nM_1)M_1.
 \end{eqnarray}
 
 Similarly,
 \begin{eqnarray*}
{\phi}_q(v^*, v_{n+1})
 & = & {\phi}_p\left(v^*, J_q\left(J_pv_n - {\lambda}_n\left(Kv_n+{\theta}_n(J_qv_n-J_qv_1)\right)\right)\right)\\
 & = & V_p\left(v^*, J_qv_n - {\lambda}_n\left(Kv_n+{\theta}_n(J_qv_n-J_qv_1)\right)\right)\\
 & \leq & V_p(v^*, J_qv_n )\\
 &&-p{\lambda}_n\left\langle J_p( J_qv_n - {\lambda}_n\left(Kv_n+{\theta}_n(J_qv_n-J_qv_1)\right))-v^*, Kv_n+{\theta}_n(J_qv_n-J_qv_1)\right\rangle \\
 &= & {\phi}_p(v^*, v_n )-p{\lambda}_n \left\langle v_n-v^*, Kv_n+{\theta}_n(J_qv_n-J_qv_1)\right\rangle\\
 & &-p{\lambda}_n\left\langle J_p( J_qv_n - {\lambda}_n\left(Kv_n+{\theta}_n(J_qv_n-J_qv_1)\right))-v_n, Kv_n+{\theta}_n(J_qv_n-J_qv_1)\right\rangle.
 \end{eqnarray*}  
 By Schwartz inequality and uniform continuity property of $J$ on bounded subsets of $E$ (Lemma \ref{l13}), we obtain
\begin{eqnarray*}
{\phi}_p(v^*, v_{n+1})
& \leq & {\phi}_p(v^*, v_n )-p{\lambda}_n \left\langle v_n-v^*, Fx_n+{\theta}_n(J_pu_n-J_pu_1)\right\rangle \\
& & + p{\lambda}_n{\psi}_1({\lambda}_nM_1)M_1 \ \mbox{(By applying inequality (\ref{e41}))}\\
& \leq &{\phi}_p(u^*, u_n )-p{\lambda}_n \left\langle u_n-u^*, Kv_n-Kv^*\right\rangle  \ \mbox{since} \ v^*\in N(K)) \\
& &-p{\lambda}_n{\theta}_n \left\langle v_n-v^*, J_qv_n-J_qv_1\right\rangle +p{\lambda}_n{\psi}_2({\lambda}_nM_2)M_2.
\end{eqnarray*}
By Lemma \ref{l31}, $p\left\langle v_n-v^*, J_qv_1-J_qv_n\right\rangle \leq {\phi}_p(v^*, v_1 )-{\phi}_p(v^*, v_n )\leq {\phi}_p(v^*, v_1 ).$ Also, since $K$ is generalized $\Phi$-strongly monotone, we have,
\begin{eqnarray} \label{e46}
{\phi}_p(v^*, v_{n+1})
 &\leq&  {\phi}_p(v^*, v_n )-p{\lambda}_n{\Phi}_2({\|v_n -v^* \|})\nonumber\\
 && +p{\lambda}_n{\theta}_n \left\langle v_n-v^*, J_qv_1-J_qv_n\right\rangle +p{\lambda}_n{\psi}_2({\lambda}_nM_2)M_2 \nonumber \\
 &\leq&  {\phi}_p(v^*, v_n )-p{\lambda}_n{\Phi}_2({\|v_n -v^* \|}) +p{\lambda}_n{\theta}_n {\phi}_p(v^*, v_1 ) +p{\lambda}_n{\psi}_2({\lambda}_nM_2)M_2.
  \end{eqnarray}
  By the uniform continuity property of $J_p$ on bounded sets of $E^*$, we have
  $$\|v_{n+1}-v_n\|=\|J_p(J_qv_{n+1})-J_p(J_qv_n)\|\leq{\psi}_2({\lambda}_nM_2),$$
  such that
  $$\|v_{n+1}-v^*\|-\|v_n-v^*\|\leq{\psi}_2({\lambda}_nM_2),$$
  which gives
  \begin{eqnarray}\label{e47}
  \|v_n-v^*\|&\geq& \|v_{n+1}-v^*\|-{\psi}_2({\lambda}_nM_2).
  \end{eqnarray}
  From Lemma \ref{l21},
  \begin{eqnarray*}
  {\|v_{n+1}-v^*\|}^p&\geq&{\phi}_p(v^*, v_{n+1})-\frac{p}{q}\|v^*\| \\
  &\geq&r-\frac{p}{q}\|u^*\| \\
  &\geq&\left({\delta}^p+\frac{p}{q}\|v^*\|\right)-\frac{p}{q}\|v^*\| \\
  &\geq&{\delta}^p.
  \end{eqnarray*}
  So, $$\|v_{n+1}-v^*\|\geq \delta.$$
  Therefore, the inequality (\ref{e47}) becomes,
 \begin{eqnarray*}  
  \|v_n-v^*\|&\geq&\delta-{\psi}_2({\lambda}_nM_2)\\
  &\geq&\frac{\delta}{2}.
 \end{eqnarray*}
 Thus,
   \begin{eqnarray}\label{e48}
 {\Phi}_2({\|v_n -v^* \|})&\geq& {\Phi}_2(\frac{\delta}{2}).
  \end{eqnarray}
  Substituting (\ref{e48}) into (\ref{e46}) gives
  \begin{eqnarray}\label{e49}
{\phi}_p(v^*, v_{n+1})
 &\leq&  {\phi}_p(v^*, v_n )-p{\lambda}_n {\Phi}_2(\frac{\delta}{2})+p{\lambda}_n{\theta}_n {\phi}_p(v^*, v_1 )\nonumber\\
 &&+p{\lambda}_n{\psi}_2({\lambda}_nM_2)M_2.
 \end{eqnarray}
 Add (\ref{e45}) and (\ref{e49}) gives 
  \begin{eqnarray*}
r<{\wedge}_p(w^*, w_{n+1})
 &\leq&  {\wedge}_p(w^*, w_n )-p{\lambda}_n \Phi(\frac{\delta}{2})+p{\lambda}_n{\theta}_n{\wedge}_p(w^*, w_1 )+p{\lambda}_n{\psi}({\lambda}_nM_0)M_0\\
 &\leq& {\wedge}_p(w^*, w_n )-p{\lambda}_n\Phi(\frac{\delta}{2})+p{\lambda}_n{\theta}_n{\wedge}_p(w^*, w_1 )+ p{\lambda}_n{\gamma}_0M_0\\
 &\leq& {\wedge}_p(w^*, w_n )-\frac{p{\lambda}_n}{2}\Phi(\frac{\delta}{2})+p{\lambda}_n{\theta}_n{\wedge}_p(w^*, w_1 )\\
 &\leq &{\wedge}_p(w^*, w_n )-\frac{p{\lambda}_n}{2}\Phi(\frac{\delta}{2})+p{\lambda}_n{\theta}_n{\wedge}_p(w^*, w_1 )\\
 &\leq &{\wedge}_p(w^*, w_n )-\frac{p{\lambda}_n}{2}\Phi(\frac{\delta}{2})+p{\lambda}_n{\wedge}_p(w^*, w_1 ) \ \mbox{(Since ${\theta}_n \in (0,1)$)}\\
 &\leq &r-\frac{p{\lambda}_n}{2}r+\frac{p{\lambda}_n}{4}r\\
 &= &r-\frac{p{\lambda}_n}{4}r<r,\\
 \end{eqnarray*}
 a contradiction. Hence, ${\wedge}_p(w^*, w_{n+1}) \leq r.$ By induction, ${\wedge}_p(w^*, w_n) \leq r  ~~ \forall  ~~ n\in \N.$ Thus, from inequality (\ref{e11}), $\left\{w_n\right\}$ is bounded.

\vskip 0.5 truecm

{\bf Part 2:} Define $A: W\rightarrow W^*$ by $Aw=(Fu-v, Kv+u), ~~ \forall~~ w=(u, v)\in W.$ We show that $\left\{w_n \right\}$ strongly converges  to a solution of $Aw=0.$ Since $A$ satisfies the range condition (Lemma \ref{e8}) and by the strict convexity of $X$ (Lemma \ref{e18}), we obtain for every $t>0$, and $w\in W$, there exists a unique $w_t\in D(A)$, where $D(A)$ is the domain of $A$ such that
$$J_p^Ww\in J_p^Ww_t+tAw_t.$$
Taking $J_tw=w_t,$ then we define a single-valued mapping $J_t : E\rightarrow D(A)$ by $J_t=(J^W_p+tA)^{-1}J_p^W$. Such a $J_t$ is called the resolvent of $A$. Therefore, by Theorem \ref{t8}, for each $n\in \N$, there exists a unique $x_n\in D(A)$ such that,

$$x_n=(J_p^W+\frac{1}{{\theta}_n}A)^{-1}J_p^Ww_1.$$

Then, setting $x_n:=(y_n, z_n)\in E\times E^*$ and $w_1:=(u_1, v_1)\in E\times E^*$,  we have 
$$(y_n, z_n)=(J^W_p+\frac{1}{{\theta}_n}A)^{-1}J^W_p(u_1, v_1),$$
which is equivalent to
$$(J_p^W+\frac{1}{{\theta}_n}A)(y_n, z_n)=J_p^W(u_1, v_1).$$
Since $A(y_n, z_n)=(Fy_n-z_n, Kz_n+y_n)$, then,
\begin{eqnarray*}
J_py_n+\frac{1}{{\theta}_n}(Fy_n-z_n)&=&J_pu_1,\\
J_qz_n+\frac{1}{{\theta}_n}(Kz_n+y_n)&=&J_qv_1,
\end{eqnarray*}
and these lead to
\begin{equation}\label{ee52}
{\theta}_n(J_py_n-J_pu_1)+Fy_n-z_n=0,
\end{equation}
\begin{equation}\label{ee53}
{\theta}_n(J_qz_n-J_qv_1)+Kz_n+y_n=0.
\end{equation}

Notice that the sequences $\left\{y_n\right\}$ and $\left\{z_n\right\}$ are bounded because they are convergent sequences by Theorem \ref{t8}. Moreover, by Theorem \ref{t8}, $\lim x_n \in A^{-1}0$. Let $y_n\rightarrow u^*$ and $z_n\rightarrow v^*,$ then $u^*$ in $E$ solves the equation $u+KFu=0$ if and only if $x^*=(u^*, v^*)$ is a solution of $Ax=0$ in $W$ for $v^*=Fu^* \in E^*.$  The implication is that 
\begin{eqnarray*}
Fu^*-v^*=0,\\
Kv^*+u^*=0.
\end{eqnarray*}
Following the same arguments as in part 1, we get,
 \begin{equation}\label{ee50}
{\phi}_p(y_n, u_{n+1}) \leq  {\phi}_p(y_n, u_n )-p{\lambda}_n \left\langle u_n-y_n, Fu_n-v_n+{\theta}_n(Ju_n-Ju_1)\right\rangle+p{\lambda}_n{\psi}_1({\lambda}_nM_1)M_1 
\end{equation}
and
 \begin{equation}\label{ee51}
{\phi}_p(z_n, v_{n+1}) \leq  {\phi}_p(z_n, v_n )-p{\lambda}_n \left\langle v_n-z_n, Kv_n+u_n+{\theta}_n(J_qv_n-J_qv_1)\right\rangle+p{\lambda}_n{\psi}_2({\lambda}_nM_2)M_2.
\end{equation}
By Theorem \ref{t3}, Lemma \ref{l21} and Eq. (\ref{ee52}), the generalized $\Phi$-strongly monotonicity of $F$ is used to obtain for some $p>1,$
\begin{eqnarray*}
&&\left\langle u_n-y_n,Fu_n-v_n+{\theta}_n(J_pu_n-J_pu_1)\right\rangle\\
&= & \left\langle x_n-y_n, Fu_n-v_n+{\theta}_n(J_pu_n-J_py_n+J_py_n-J_pu_1)\right\rangle\\
&= & {\theta}_n\left\langle u_n-y_n,J_pu_n-J_py_n\right\rangle+ \left\langle u_n-y_n, Fu_n-v_n +{\theta}_n(J_py_n-J_pu_1)\right\rangle\\
&= &{\theta}_n\left\langle u_n-y_n,J_pu_n-J_py_n\right\rangle+\left\langle u_n-y_n, Fu_n-v_n-(Fy_n-z_n) \right\rangle\\
&\geq & {\theta}_ng(\|u_n-y_n \|) +\Phi(\|u_n-y_n \|)+\left\langle u_n-y_n, z_n-v_n \right\rangle\\
&\geq & \frac{1}{p}{\theta}_n {\phi}_p(y_n, u_n )+\left\langle u_n-y_n, z_n-v_n \right\rangle \\
\end{eqnarray*}

This makes the inequality (\ref{ee50}) to become
\begin{equation}\label{ee54}
{\phi}_p(y_n, u_{n+1}) \leq  (1-{\lambda}_n{\theta}_n){\phi}_p(y_n, u_n )-p{\lambda}_n\left\langle u_n-y_n, z_n-v_n \right\rangle +p{\lambda}_n{\psi}_1({\lambda}_nM_1)M_1.
\end{equation}
From Lemma \ref{l31}, we obtain that
\begin{eqnarray}\label{ee55}
{\phi}_p(y_n, u_n )&\leq &  {\phi}_p(y_{n-1}, u_n )-p\left\langle y_n-u_n, J_py_{n-1}-J_py_n\right\rangle \nonumber \\ 
 &=& {\phi}_p(y_{n-1}, u_n )+p\left\langle u_n-y_n, J_py_{n-1}-J_py_n \right\rangle \nonumber \\
& \leq& {\phi}_p(y_{n-1}, u_n )+\|J_py_{n-1}-J_py_n\|\|u_n-y_n\|.
\end{eqnarray}
Let $R > 0$ such that $\|x_1\| \leq R, \|y_n\| \leq R$ for all $n \in  \N$. Then the estimates below follows from (\ref{ee52}),
$$J_py_{n-1}-J_py_n+\frac{1}{{\theta}_n}\left(Fy_{n-1}-z_{n-1}-(Fy_n-z_n\right)=  \frac{{\theta}_{n-1}-{\theta}_n}{{\theta}_n}\left(J_pu_1-J_py_{n-1}\right).$$
 Taking the duality pairing of each side of this equation with respect to $y_{n-1}-y_n$ and using the generalized $\Phi$-strongly monotonicity property of $F$, then
$$\left\langle J_py_{n-1}-J_py_n, y_{n-1}-y_n\right\rangle \leq  \frac{{\theta}_{n-1}-{\theta}_n}{{\theta}_n}\| J_pu_1-J_py_{n-1}\|\| y_{n-1}-y_n\|,$$
which gives,
\begin{equation}\label{ee56}
 \|J_py_{n-1}-J_py_n\| \leq \left( \frac{{\theta}_{n-1}}{{\theta}_n}-1\right)\|J_py_{n-1}-J_pu_1\|.
\end{equation}
Using (\ref{ee55}) and (\ref{ee56}), the inequality (\ref{ee50}) becomes
 \begin{eqnarray}\label{ee57}
{\phi}_p(y_n, u_{n+1}) &\leq&  (1-{\lambda}_n{\theta}_n){\phi}_p(y_{n-1}, u_n)+C_1\left( \frac{{\theta}_{n-1}}{{\theta}_n}-1\right)\nonumber\\
&&-p{\lambda}_n\left\langle u_n-y_n, z_n-v_n \right\rangle + p{\lambda}_n{\psi}_1({\lambda}_nM_1)M_1,
\end{eqnarray}
for some constant $C_1 > 0.$ Similar analysis gives that
 \begin{eqnarray}\label{ee58}
{\phi}_p(z_n, v_{n+1}) &\leq&  (1-{\lambda}_n{\theta}_n){\phi}_p(z_{n-1}, v_n)+ C_2\left( \frac{{\theta}_{n-1}}{{\theta}_n}-1\right)\nonumber\\
&&-p{\lambda}_n\left\langle v_n-z_n, u_n-y_n \right\rangle + p{\lambda}_n{\psi}_2({\lambda}_nM_2)M_2,
\end{eqnarray}
for some constant $C_2 > 0$. Since ${\psi}:={\psi}_1+{\psi}_2, ~M_0:=M_1+M_2$ and ${\psi}({\lambda}_nM_0) \leq{\gamma}_0,$ adding (\ref{ee56}) and (\ref{ee58}) generates
$$\wedge(x_n, w_{n+1}) \leq (1-{\lambda}_n{\theta}_n)\wedge(x_{n-1}, w_n)+C\left( \frac{{\theta}_{n-1}}{{\theta}_n}-1\right)+ p{\lambda}_n{\gamma}_0M_0,$$
 where $C:=C_1+C_2 > 0.$ By Lemma \ref{l11}, $\phi(x_{n-1}, w_n )\rightarrow 0$ as $n\rightarrow \infty$ and using Lemma \ref{l15}, we have that $w_n-x_{n-1}\rightarrow 0$ as $n\rightarrow \infty$. Since by Theorem \ref{t8}, $x_n\rightarrow w^* \in N(A)$, we obtain that $w_n\rightarrow w^*$ as $n\rightarrow \infty$. But $w_n=(u_n, v_n)$ and $w^*=(u^*, v^*)$, this implies that $u_n\rightarrow u^*$ with $u^*$ the solution of the Hammerstein equation. 
 \end{proof}

\begin{corollary}\label{e16}
Let $E$ be a uniformly smooth and uniformly convex real Banach space with the dual space $E^*$. Suppose $F: E\rightarrow E^*$ and $K: E^*\rightarrow E$ are bounded and strongly monotone mappings. Define $\left\{u_n\right\}$ and $\left\{v_n\right\}$ iteratively for arbitrary $u_1 \in E$ and $v_1\in E^*$ by
\begin{equation}
u_{n+1} = J_q\left(J_pu_n - {\lambda}_n\left(Fu_n-v_n+{\theta}_n(J_pu_n-J_pu_1)\right)\right), n \in \N,
\end{equation}
\begin{equation}
v_{n+1} = J_p\left(J^*_qv_n - {\lambda}_n\left(Kv_n+u_n+{\theta}_n(J^*_qv_n-J^*_qv_1)\right)\right), n \in \N,
\end{equation}
where $J_p:E\rightarrow E^*$ is the generalized duality mapping with the inverse, $J_q:E^*\rightarrow E$ and the real sequences $\left\{ {\lambda}_n\right\}$ and  $\left\{ {\theta}_n\right\}$ in $(0,1)$ are such that,
\begin{itemize}
	\item[(i)] $\lim {\theta}_n =0$ and $\left\{ {\theta}_n\right\}$ is decreasing;
	\item[(ii)]$ \displaystyle\sum_{n=1}^{\infty} {\lambda}_n{\theta}_n=\infty$;
	\item[(iii)]$\displaystyle \lim_{n\rightarrow \infty}\left(({\theta}_{n-1}/{\theta}_n)-1\right)/{{\lambda}_n{\theta}_n}=0,~~ \displaystyle\sum_{n=1}^{\infty} {\lambda}_n <\infty$.
\end{itemize}	
Suppose that $u + KFu=0$ has a solution in $E$. There exists a real constant ${\gamma}_0>0$ with ${\psi}({\lambda}_nM)\leq{\gamma}_0, \ \ n \in \N$ for some constant $M>0$. Then, the sequence $\left\{u_n \right\}$ converges strongly to the solution of $0=u+KFu$.
 \end{corollary}
\begin{proof}
Define ${\Phi}_1(\|u_1-u_2\|):=k_1{\|u_1-u_2\|}^2$ and ${\Phi}_2(\|v_1-v_2\|):=k_2{\|v_1-v_2\|}^2$ for some constants $k_1, k_2\in (0,1)$ and let $W := E\times E^*$ with norm ${\|w \|}^2_W :=   {\| u\|}^2_{E} + {\| v \|}^2_{E^*}  ~ \forall ~ w = \left(u, v\right) \in W.$  The result follows from Theorem \ref{e14}.
\end{proof}
\begin{corollary}\label{h15}
Chidume and Idu \cite{b28}. Let $E$ be a uniformly convex and uniformly smooth real Banach space and $F : E \rightarrow E^*$, $K : E^* \rightarrow E$ be maximal monotone and bounded maps, respectively. For $(x_1, y_1), (u_1, v_1) \in E\times E^*$, define the sequences $\left\{u_n\right\}$ and $\left\{v_n\right\}$ in $E$ and $E^*$ respectively, by  
\begin{equation}\label{h16}
u_{n+1} = J^{-1}\left(Ju_n-{\lambda}_n(Fu_n-v_n)-{\lambda}_n{\theta}_n(Ju_n-Jx_1)\right),~~ n \in \N,
\end{equation}
\begin{equation}\label{h161}
v_{n+1} = J\left(J^{-1}v_n-{\lambda}_n(Kv_n+u_n)-{\lambda}_n{\theta}_n(J^{-1}v_n-J^{-1}y_1)\right),~~ n \in \N,
\end{equation}
where $\left\{ {\lambda}_n\right\}$ and  $\left\{ {\theta}_n\right\}$ are real sequences in $(0,1)$ satisfying the following conditions:
\begin{itemize}
  \item[(i)]$ \displaystyle\sum_{n=1}^{\infty} {\lambda}_n {\theta}_n=\infty$,
	\item[(ii)]${\lambda}_nM^*_0\leq {\gamma}_0{\theta}_n$; ${\delta}_E^{-1}({\lambda}_nM^*_0)\leq{\gamma}_0{\theta}_n$,
	\item[(iii)]$\frac{{\delta}_E^{-1}\left(\frac{{\theta}_{n-1}-{\theta}_n}{{\theta}_n} K\right)}{{{\lambda}_n \theta}_n}\rightarrow 0$; $\frac{{\delta}_{E^*}^{-1}\left(\frac{{\theta}_{n-1}-{\theta}_n}{{\theta}_n} K\right)}{{{\lambda}_n \theta}_n}\rightarrow 0$ as $n\rightarrow \infty$,
	\item[(iv)]$\frac{1}{2}\frac{{\theta}_{n-1}-{\theta}_n}{{\theta}_n} K \in (0, 1)$,
\end{itemize}
for some constants $M^*_0>0$ and ${\gamma}_0>0$, where ${\delta}_E:(0, \infty)\rightarrow (0, \infty)$ is the modulus of convexity of $E$ and $K:=4RL \sup\left\{\|Jx-Jy\|: \|x\|\leq R, \|y\|\leq R \right\}+1,~~ x, y \in E,~~ R>0$. Assume that the equation $u+KFu = 0$ has a solution. Then the sequences $\left\{u_n\right\}^{\infty}_{n=1}$ and $\left\{v_n\right\}^{\infty}_{n=1}$ converge strongly to $u^*$ and $v^*,$ respectively, where $u^*$ is the solution of $u + KFu = 0$ with $v^* = Fu^*$.
\end{corollary}

\begin{proof}
From Lemma \ref{l17}, we see that $T : E \times E^* \rightarrow E^* \times E$ defined by $T(u, v) = (Ju-Fu+v, J^{-1}v-Kv-u)$ for all $(u, v) \in E \times E^*$ is $J$-pseudocontractive and $A := (J-T)$ is maximal monotone. Therefore, the iterative sequences (\ref{h16}) and (\ref{h161}) are respectively equivalent to 
 \begin{equation}
u_{n+1} = J^{-1}\left(Ju_n - {\lambda}_n\left(Fu_n+{\theta}_n(Ju_n-Jx_1)\right)\right), n \in \N \ \mbox{and}
\end{equation}
 \begin{equation}
v_{n+1} = J\left(J^{-1}v_n - {\lambda}_n\left(Kv_n+{\theta}_n(J^{-1}v_n-J^{-1}y_1)\right)\right), n \in \N,
\end{equation}
where $J:E\rightarrow E^*$ is the normalized duality mapping with the inverse, $J^{-1}:E^*\rightarrow E.$ Hence, the result follows from Theorem \ref{e14}.
\end{proof}
\begin{remark}
 Prototype for our iteration parameters in Theorem \ref{e14} are, ${\lambda}_n=\frac{1}{(n+1)^a}$ and ${\theta}_n=\frac{1}{(n+1)^b}$, where $0 < b < a$ and $a < 1$.
\end{remark}

\begin{conclusion} 
We have considered the class of generalized $\Phi$-strongly monotone mappings in Banach spaces. This is the class of monotone-type mappings such that if a solution of the equation $0\in Ax$ exists, it is necessarily unique. Our results generalize and improve the recent and important results of Chidume and Idu \cite{b28}. Also, our results show extention and application of the main results of Aibinu and Mewomo \cite{b9, b4}.
\end{conclusion}

\vskip 0.3 truecm
\textbf{Acknowledgment:\\} The first author acknowledges with thanks the bursary and financial support from Department of Science and Technology and National Research Foundation, Republic of South Africa Center of Excellence in Mathematical and Statistical Sciences (DST-NRF CoE-MaSS) Doctoral Bursary. Opinions expressed and conclusions arrived at are those of the authors and are not necessarily to be attributed to the CoE-MaSS.


\begin{thebibliography}{99}
\bibitem{b9}  M. O. Aibinu and O. T. Mewomo, \emph {Algorithm for Zeros of monotone maps in Banach spaces}, Proceedings of Southern  Africa Mathematical Sciences Association  (SAMSA2016) Annual Conference, 21-24 November, 2016, University of Pretoria, South Africa, (2017), 35-44.
\bibitem{b4} M. O. Aibinu and O. T. Mewomo, \emph {Strong convergence theorems for strongly monotone mappings in Banach spaces}, Boletim da Sociedade Paranaense de Matem$\acute{a}$tica, Boletim da Societade Paranaense de Matema'tica, 39 (1), (2021), 169-187.
\bibitem{aibinu}	M. O. Aibinu, S.C. Thakur and M. Moyo, \emph {Algorithm for solutions of nonlinear equations of strongly monotone type and applications to convex minimization and variational inequality problems}, Abstract and Applied Analysis, Vol. 2020, Article ID: 6579720, (2020), 1-11.
\bibitem{b1}  Ya. Alber, \emph{Metric and generalized projection operators in Banach spaces: properties and applications}, In: Kartsatos
\bibitem{e7} Y. Alber and I. Ryazantseva, \emph { Nonlinear Ill posed problems of monotone type}, Springer, London, (2006).
\bibitem{b15} H. Br$\acute{e}$zis and  F. E. Browder, \emph{ Some new results about Hammerstein equations}, Bull. Amer. Math. Soc., 80, (1974), 567-572.
\bibitem{b16} H. Br$\acute{e}$zis and  F. E. Browder, \emph{ Existence theorems for nonlinear integral equations of Hammerstein type}, Bull. Amer. Math. Soc., 81, (1975), 73-78.
\bibitem{b22} H. Br$\acute{e}$zis and  F. E. Browder, \emph { Nonlinear integral equations and system of Hammerstein type}, Adv. Math. 18, (1975), 115-147.
\bibitem{b38} F. E. Browder and C. P. Gupta, \emph{ Monotone operators and nonlinear integral equations of Hammerstein type}, Bull. Amer. Math. Soc., 75, (1969), 1347-1353.
\bibitem{b19} R. Sh. Chepanovich, \emph{ Nonlinear Hammerstein equations and fixed points}, Publ. Inst. Math. (Beograd) (N.S.), 35, (49), (1984), 119-123.
\bibitem{e33} C. E. Chidume and N. Djitte, \emph{Strong convergence theorems for zeros of bounded maximal monotone nonlinear operators}, Abstract and Applied Analysis, Volume 2012, Article ID 681348, 19 pages, doi:10.1155/2012/681348
\bibitem{b28} C. E. Chidume and K. O. Idu, \emph {Approximation of zeros of bounded maximal monotone mappings, solutions of Hammerstein integral equations and convex minimization problems}, Fixed Point Theory and Applications, 2016, (97), (2016), DOI 10.1186/s13663-016-0582-8.
 \bibitem{Chidume1} C.E. Chidume, M.O. Nnakwe, A. Adamu,  \emph {A strong convergence theorem for generalized $\Phi$-strongly monotone maps with applications}, Fixed Point Theory Appl.,  2019: 11, (2019), 19 pages.
\bibitem{b3} I. Cioranescu, \emph {Geometry of Banach spaces, duality mappings and nonlinear problems}, Kluwer Academic Publishers Group, Dordrecht, (1990).
\bibitem{b20} D. G. De Figueiredo and C. P. Gupta, \emph{On the variational method for the existence of solutions of nonlinear equations of Hammerstein type}, Proc. Amer. Math. Soc., 40, (1973), 470-476.
 \bibitem{Djitte1} N. Djitte, J. T. Mendy and T. M. M. Sow, \emph {Computation of zeros of monotone type mappings: on Chidume's open problem}, J. Aust. Math. Soc., 108 (2), (2020), 278-288. 
\bibitem{b14} V.  Dolezale, \emph{ Monotone operators and applications in control and network theory}, Studies in Automation and Control, Elsevier Scientific, New York, USA, 2, (1979).
\bibitem{e2} A. Hammerstein, \emph{Nichtlineare integralgleichungen nebst anwendungen}, Acta Mathematica, 54, ( 1), (1930), 117-176.
\bibitem{r16} S. Kamimura and W. Takahashi, \emph {Strong convergence of a proximal-type algorithm in Banach a space}, SIAM J. Optim., 13, (2002), 938-945.
 \bibitem{b30} K. Kido, \emph { Strong convergence of resolvents of monotone operators in Banach spaces}, Proc. Am. Math. Soc. 103 (3), (1988), 755-7588.
 \bibitem{b34} B. T. Kien,\emph { The normalized duality mappings and two related characteristic properties of a uniformly convex Banach space}, Acta Mathematica Vietnamica, 27, (1), (2002), 53-67.
\bibitem{b32} F. Kohsaka and W. Takahashi, \emph {Existence and approximation of fixed points of firmly nonexpansive type mappings in Banach spaces}, SIAM J. on Optim., 19, (2), (2008), 824-835.
\bibitem{e35} S. Y. Matsushita and W. Takahashi, \emph {Weak and strong convergence theorems for relatively nonexpansive mappings in Banach spaces, Fixed Point Theory and Applications}, 2004:1 (2004) 37–47, DOI 10.1155/S1687182004310089
\bibitem{e4} D. Pascali and S. Sburlan, \emph { Nonlinear mappings of monotone type}, Editura Academiae, Bucharest, Romania, (1978).
\bibitem{b2}  W. Takahashi, \emph{Convex analysis and approximation fixed points}; Yokohama Publishers, Yokohama, Japanese {2000}.
\bibitem{e31}  W. Takahashi, \emph{ Fixed point theory and its applications. In Nonlinear functional analysis}., Yokohama Publishers, (2000).
 \bibitem{d26} H. K. Xu, \emph {Inequalities in Banach spaces with applications}, Nonlinear Anal., 16 (12), (1991), 1127-1138.
 \bibitem{bh1} H. K. Xu, \emph { Iterative algorithms for nonlinear operators}, Journal of the London Mathematical Society II,  66, 1, (2002), 240-256.
\bibitem{b7} Z. B. Xu and G. F. Roach, \emph {Characteristic inequalities of uniformly convex and uniformly smooth Banach spaces}, J. Math. Anal. Appl. 157, (1991), 189-210.
\bibitem{Yekini2} S. Yekini, \emph{Convergence results of forward-backward algorithms for sum of monotone operators in Banach spaces}, Results Math., 74:138, (2019), 24 pages.
\bibitem {b37} C. Z$\check{a}$linescu, \emph {On uniformly convex functions},  J. Math. Anal. Appl. 95, (1983), 344-374.
\bibitem{e1} E. Zeidler, \emph{Nonlinear Functional Analysis and Its Applications, Part II: Monotone operators}, Springer, New York, USA, (1985).
\end{thebibliography}
\end{document}